\documentclass[11pt]{amsart}
\usepackage[utf8]{inputenc}
\usepackage[T1]{fontenc}
\usepackage{indentfirst, amssymb, amsmath, amsthm, mathrsfs, setspace, enumerate}
\usepackage{xcolor}
\definecolor{purple}{rgb}{0.5,0,0.5}
\usepackage[colorlinks=true,linkcolor=purple,citecolor=blue,urlcolor=magenta]{hyperref}
\newtheorem{theo}{Theorem}[section]
\newtheorem{lem}{Lemma}[section]
\newtheorem{prop}{Proposition}[section]
\newtheorem{cor}{Corollary}[section]
\newtheorem{rem}{Remark}[section]
\newtheorem{definition}{Definition}[section]

\numberwithin{equation}{section}

\begin{document}

\title[Directional Coefficient Inequalities for Strongly Starlike Mappings]{Directional Coefficient Inequalities for Strongly Starlike Mappings in Complex Banach Spaces}

\author[P. Das and N. Sarkar]{Pradip Das and Nabadwip Sarkar}

\address{Pradip Das, Department of Mathematics, Raiganj University, Raiganj, West Bengal-733134, India.}
\email{pradipsmath@gmail.com}

\address{Nabadwip Sarkar, Amity School of Applied Sciences, Amity University Mumbai, Panvel, Navi Mumbai, Maharashtra-410206, India}
\email{nsarkar@mum.amity.edu, nabadwipsarkar52@gmail.com}

\makeatletter
\@namedef{subjclassname@2020}{\textup{2020} Mathematics Subject Classification}
\makeatother


\begin{abstract}
We study a class of strongly starlike mappings of order $\alpha$ on the open unit ball of a complex Banach space under a scalar-profile condition. Using a slice reduction along one-dimensional complex subspaces, we prove a coefficient correspondence theorem identifying the directional coefficients with the Taylor coefficients of an associated one-variable strongly starlike function. We further prove a lifting theorem showing that every one-variable strongly starlike function can be realized by a scalar-profile Banach-space mapping. Consequently, the extremal problems for polynomial directional coefficient functionals coincide exactly with their one-variable counterparts. We apply this correspondence to obtain bounds for $|A_2A_3 - A_4|$, $|A_2A_4 - A_3^2|$, and $|A_3 - A_2^2|$, together with the corresponding estimates for $|A_4|$ and $|A_5|$. We also establish a parameterized Fekete--Szeg\H{o} estimate and a sharp bound for the second Hankel determinant of directional logarithmic coefficients.
\end{abstract}

\maketitle

\section{Introduction}

Let $\mathbb{D} := \{z \in \mathbb{C} : |z| < 1\}$ denote the open unit disk. Let $\mathcal{A}$ be the family of analytic functions $f$ on $\mathbb{D}$ normalized by $f(0)=0$, $f'(0)=1$, so that
\[
f(z) = z + \sum_{n=2}^{\infty} a_n z^n.
\]
The Hankel determinant $H_{q,n}(f)$ of order $q$ and shift $n$ is defined by
\begin{equation}
H_{q,n}(f) :=
\begin{vmatrix}
a_n & a_{n+1} & \cdots & a_{n+q-1}\\
a_{n+1} & a_{n+2} & \cdots & a_{n+q}\\
\vdots & \vdots & \ddots & \vdots\\
a_{n+q-1} & a_{n+q} & \cdots & a_{n+2(q-1)}
\end{vmatrix}.
\end{equation}
For $\alpha \in (0,1]$, a function $f \in \mathcal{A}$ is strongly starlike of order $\alpha$, written $f \in \mathcal{S}^*_\alpha$, if
\[
\left|\arg\left(\frac{zf'(z)}{f(z)}\right)\right| < \alpha \frac{\pi}{2}
\qquad (z \in \mathbb{D}).
\]
Cho, Kowalczyk, Kwon, Lecko, and Sim \cite{cho2017} obtained sharp bounds for coefficient functionals occurring in the expansion of the third-order Hankel determinant over $\mathcal{S}^*_\alpha$, including
\[
|a_2 a_3 - a_4|, \qquad |a_2 a_4 - a_3^2|, \qquad |a_3 - a_2^2|.
\]

Extending such results to locally biholomorphic mappings on the open unit ball $\mathbb{B}$ of a complex Banach space $X$ requires multilinear operator formulations, because vector division is not defined. Relevant earlier work includes Suffridge \cite{suffridge}, Gong \cite{gong}, Graham--Kohr \cite{graham}, and Pfaltzgraff \cite{pfaltzgraff}. Xu, D\L{}ugosz, and Liczberski \cite{XuDlugosz} studied third-order Hankel determinants for quasi-convex mappings in complex Banach spaces. Further results on parameterized operator functionals and several complex variables appear in \cite{elin2024,hamada2023,xu2024}.

In this paper we introduce a Banach space analogue of strongly starlike mappings under a scalar-profile restriction. We prove a slice correspondence for the scalar-profile class, showing that the directional coefficients of any scalar-profile mapping coincide with the Taylor coefficients of the associated slice function. We then prove a lifting theorem in the converse direction: every one-variable strongly starlike function arises as a slice of some scalar-profile Banach-space mapping. Together these two results give an equality of the extremal problems for polynomial directional coefficient functionals: the supremum of any such functional over the Banach-space class equals the supremum over the one-variable class. Thus sharpness statements, and not only numerical inequalities, can be transferred from the one-variable theory to the Banach space setting.

Coefficient inequalities for subclasses of strongly starlike mappings in several complex variables were studied by Xu and Xu \cite{xu2018_strongly}. Their approach used direct expansion techniques. The second Hankel determinant of logarithmic coefficients for strongly starlike functions in one variable was studied by Kowalczyk and Lecko \cite{Kowalczyk2023}. The scalar-profile form considered here is a subclass of mappings in higher-dimensional function theory; the present paper gives a correspondence for this subclass.

The paper is organized as follows. Section 2 gives definitions and preliminaries. Section 3 contains the slice correspondence and the lifting theorem. Section 4 collects the one-variable results that will be used. Section 5 gives the directional coefficient inequalities. Section 6 describes the extremal mappings. Section 7 treats the second Hankel determinant of directional logarithmic coefficients.

\section{Definitions and Preliminaries}

Let $(X,\|\cdot\|)$ be a complex Banach space. Let
\[
\mathbb{B} := \{x \in X : \|x\| < 1\}
\]
be its open unit ball, and let $\partial \mathbb{B} := \{x \in X : \|x\| = 1\}$ be its boundary. Let $L(X,Y)$ denote the space of continuous linear operators from $X$ into $Y$, and let $I \in L(X,X)$ be the identity operator. For $x \in X \setminus \{0\}$, define
\begin{equation}
T(x) := \{T_x \in L(X,\mathbb{C}) : T_x(x) = \|x\|,\ \|T_x\| = 1\}.
\end{equation}
By the Hahn--Banach theorem, $T(x)$ is nonempty.

Let $H(\mathbb{B})$ be the set of all holomorphic mappings from $\mathbb{B}$ into $X$. A mapping $F \in H(\mathbb{B})$ is locally biholomorphic if $DF(x)$ has a bounded inverse $(DF(x))^{-1} \in L(X,X)$ for every $x \in \mathbb{B}$. The $n$-th Fr\'echet derivative $D^n F(0)$ is a bounded symmetric $n$-linear mapping from $X^n$ into $X$. We say that $F$ is normalized if $F(0) = 0$ and $DF(0) = I$.

For a fixed direction $x_0 \in \partial \mathbb{B}$ and a functional $T_{x_0} \in T(x_0)$, the directional coefficients $A_n$ are defined by
\begin{equation}\label{pra1}
A_1 = 1, \qquad
A_n = \frac{1}{n!} T_{x_0}\bigl(D^n F(0)(x_0^n)\bigr), \quad n \ge 2.
\end{equation}

The following lemma shows that evaluation on a one-dimensional subspace does not depend on the choice of norming functional.

\begin{lem}\label{pra2}
Let $x \in X \setminus \{0\}$. For any two norming functionals $T_x, S_x \in T(x)$ and any $v \in \operatorname{span}\{x\}$, we have
\[
T_x(v) = S_x(v).
\]
\end{lem}

\begin{proof}
Write $v = \lambda x$ for some $\lambda \in \mathbb{C}$. Since $T_x(x) = S_x(x) = \|x\|$,
\[
T_x(v) = T_x(\lambda x) = \lambda T_x(x) = \lambda \|x\|
\]
and
\[
S_x(v) = S_x(\lambda x) = \lambda S_x(x) = \lambda \|x\|.
\]
Thus $T_x(v) = S_x(v)$.
\end{proof}

\begin{definition}\label{pra3}
Let $F \in H(\mathbb{B})$ be normalized and locally biholomorphic, and suppose that
\[
F(x) = g(x)x,
\]
where $g : \mathbb{B} \to \mathbb{C}$ is holomorphic. We require
\begin{equation}\label{pra4}
g(x) \neq 0 \qquad (x \in \mathbb{B}).
\end{equation}
Since $g(0) = 1$ by normalization, this condition is nontrivial only for $x \neq 0$. For $\alpha \in (0,1]$, we say that $F$ belongs to the scalar-profile strongly starlike class of order $\alpha$, denoted by $\mathcal{S}^*_\alpha(\mathbb{B})$, if for every $x \in \mathbb{B} \setminus \{0\}$ and every $T_x \in T(x)$,
\begin{equation}\label{pra5}
\left|
\arg\left(T_x\bigl((DF(x))^{-1}F(x)\bigr)\right)
\right|
< \alpha \frac{\pi}{2}.
\end{equation}
\end{definition}

\begin{rem}\label{pra6}
Condition \eqref{pra4} ensures that the argument in \eqref{pra5} is defined. Indeed, for $F(x) = g(x)x$ with $x \in \mathbb{B} \setminus \{0\}$, differentiation gives
\[
DF(x)x = g(x)x + [Dg(x)x]x = \bigl(g(x) + Dg(x)x\bigr)x.
\]
Since $F$ is locally biholomorphic, $DF(x)$ is invertible, so the eigenvalue $g(x) + Dg(x)x$ is nonzero. Hence
\[
(DF(x))^{-1}x = \frac{1}{g(x) + Dg(x)x}\,x,
\]
and therefore
\[
(DF(x))^{-1}F(x) = g(x)(DF(x))^{-1}x = \frac{g(x)}{g(x) + Dg(x)x}\,x = \lambda(x)\,x
\]
with $\lambda(x) = g(x)/(g(x) + Dg(x)x)$. Since $g(x) \neq 0$ by \eqref{pra4} and the denominator is nonzero, we have $\lambda(x) \neq 0$. Thus $T_x((DF(x))^{-1}F(x)) = \lambda(x)\|x\|$ is nonzero, and its argument is defined.

When $X = \mathbb{C}$ and $\mathbb{B} = \mathbb{D}$, a normalized mapping becomes an analytic function
\[
f(z) = z + \sum_{n=2}^{\infty} a_n z^n.
\]
The unique norming functional is $T_z(v) = v|z|/z$. Hence
\[
T_z\bigl((DF(z))^{-1}f(z)\bigr)
= T_z\left(\frac{f(z)}{f'(z)}\right)
= |z|\frac{f(z)}{z f'(z)}.
\]
Taking arguments, the positive factor $|z|$ cancels, and we recover the classical definition
\[
\left|\arg\left(\frac{z f'(z)}{f(z)}\right)\right| < \alpha \frac{\pi}{2}.
\]
This is the definition of $\mathcal{S}^*_\alpha$.

For the scalar-profile subclass, the vector $(DF(x))^{-1}F(x)$ lies in $\operatorname{span}\{x\}$. Since every $T_x \in T(x)$ satisfies $T_x(x) = \|x\|$, Lemma \ref{pra2} shows that the evaluation on this one-dimensional subspace does not depend on the choice of $T_x$. Thus condition \eqref{pra5} is independent of the particular norming functional selected.
\end{rem}

\begin{rem}\label{pra7}
If $F(x) = g(x)x$ is normalized, then $DF(0) = I$. Differentiating at $x = 0$ gives
\[
DF(0)v = g(0)v + [Dg(0)v]\,0 = g(0)v.
\]
Since $DF(0) = I$, we obtain $g(0) = 1$. Thus the normalization $g(0) = 1$ is automatic rather than an additional assumption. It is retained below as a convenient explicit normalization.
\end{rem}

\section{Slice Reduction and Coefficient Correspondence}

Let $F(x) = g(x)x$ be normalized on $\mathbb{B}$. Fix $x_0 \in \partial \mathbb{B}$ and $T_{x_0} \in T(x_0)$. Define the slice function $f : \mathbb{D} \to \mathbb{C}$ by
\[
f(\xi) = T_{x_0}\bigl(F(\xi x_0)\bigr) = g(\xi x_0)\xi.
\]
Since $T_{x_0}(x_0) = 1$ and $F(\xi x_0) = g(\xi x_0)\xi x_0$, this definition is independent of the particular choice of $T_{x_0} \in T(x_0)$.

\begin{prop}\label{pra8}
Let $F(x) = g(x)x$ be a normalized holomorphic mapping on $\mathbb{B}$. Then for every $x_0 \in \partial \mathbb{B}$ and every $n \in \mathbb{N}$,
\begin{equation}\label{pra9}
\frac{1}{n!} D^n F(0)(x_0^n)
=
\frac{1}{(n-1)!}\bigl(D^{n-1} g(0)(x_0^{n-1})\bigr) x_0.
\end{equation}
\end{prop}

\begin{proof}
Let $G(x) = x$ be the identity mapping. Then $DG(x) = I$ and $D^k G(x) = 0$ for $k \ge 2$. By the Leibniz rule,
\[
D^n F(0)(x_0^n)
=
\sum_{k=0}^{n}\binom{n}{k}
D^{n-k} g(0)(x_0^{n-k}) D^k G(0)(x_0^k).
\]
Since $G(0) = 0$ and $D^k G(0) = 0$ for $k \ge 2$, only the term with $k = 1$ remains. Hence
\[
D^n F(0)(x_0^n)
=
\binom{n}{1} D^{n-1} g(0)(x_0^{n-1}) DG(0) x_0
=
n D^{n-1} g(0)(x_0^{n-1}) x_0.
\]
Dividing by $n!$ gives \eqref{pra9}.
\end{proof}

\begin{lem}\label{pra10}
Let $F(x) = g(x)x$ be normalized and locally biholomorphic on $\mathbb{B}$, satisfying \eqref{pra4}. Let $x_0 \in \partial \mathbb{B}$ and let $f(\xi) = g(\xi x_0)\xi$. Then for every $\xi \in \mathbb{D} \setminus \{0\}$,
\begin{equation}\label{pra11}
T_{x_0}\bigl((DF(\xi x_0))^{-1} F(\xi x_0)\bigr) = \frac{f(\xi)}{f'(\xi)}.
\end{equation}
\end{lem}

\begin{proof}
Differentiating $F(x) = g(x)x$ at $x = \xi x_0$ gives
\[
DF(\xi x_0)v = g(\xi x_0)v + [Dg(\xi x_0)v]\xi x_0.
\]
Since $F$ is locally biholomorphic, $DF(\xi x_0)$ is invertible. Let
\[
v = (DF(\xi x_0))^{-1} F(\xi x_0).
\]
Since $T_{x_0}(x_0) = 1$, the kernel and span form a direct sum. Indeed, if $v \in \ker T_{x_0} \cap \operatorname{span}\{x_0\}$, then $v = \mu x_0$ for some $\mu \in \mathbb{C}$ and $0 = T_{x_0}(v) = \mu T_{x_0}(x_0) = \mu$, so $v = 0$. Moreover, every $v \in X$ can be written as
\[
v = \bigl(v - T_{x_0}(v)x_0\bigr) + T_{x_0}(v)x_0,
\]
with $v - T_{x_0}(v)x_0 \in \ker T_{x_0}$ and $T_{x_0}(v)x_0 \in \operatorname{span}\{x_0\}$. Thus $X = \ker T_{x_0} \oplus \operatorname{span}\{x_0\}$, and we may write $v = \lambda x_0 + w$ with $w \in \ker T_{x_0}$ and $\lambda \in \mathbb{C}$.

Substituting into
\[
DF(\xi x_0)v = F(\xi x_0) = g(\xi x_0)\xi x_0
\]
gives
\[
g(\xi x_0)w + [Dg(\xi x_0)w]\xi x_0
+ \lambda\bigl(g(\xi x_0) + \xi Dg(\xi x_0)x_0\bigr)x_0
= g(\xi x_0)\xi x_0.
\]

Comparing the two components with respect to the decomposition $X = \ker T_{x_0} \oplus \operatorname{span}\{x_0\}$: the term $g(\xi x_0)w$ lies in $\ker T_{x_0}$, since $w \in \ker T_{x_0}$ and scalar multiplication preserves the kernel. All remaining terms are multiples of $x_0$, hence lie in $\operatorname{span}\{x_0\}$. Projecting onto $\ker T_{x_0}$ gives
\[
g(\xi x_0)w = 0.
\]
Since $g(\xi x_0) \neq 0$ by \eqref{pra4}, we obtain $w = 0$. In particular, $Dg(\xi x_0)w = 0$. Projecting now onto $\operatorname{span}\{x_0\}$ and using $w = 0$,
\[
\lambda\bigl(g(\xi x_0) + \xi Dg(\xi x_0)x_0\bigr) = g(\xi x_0)\xi.
\]

On the other hand, direct differentiation of $f(\xi) = g(\xi x_0)\xi$ gives
\[
f'(\xi) = g(\xi x_0) + \xi Dg(\xi x_0)x_0.
\]
Moreover,
\[
DF(\xi x_0)x_0
=
\bigl[g(\xi x_0) + \xi Dg(\xi x_0)x_0\bigr]x_0
=
f'(\xi)x_0,
\]
so $f'(\xi)$ is the eigenvalue of the invertible operator $DF(\xi x_0)$ on the radial direction $x_0$, and hence $f'(\xi) \neq 0$. Therefore
\[
\lambda = \frac{\xi g(\xi x_0)}{g(\xi x_0) + \xi Dg(\xi x_0)x_0} = \frac{f(\xi)}{f'(\xi)}.
\]
Since $T_{x_0}(v) = \lambda$, the result follows.
\end{proof}

\begin{theo}\label{pra12}
Let $F \in \mathcal{S}^*_\alpha(\mathbb{B})$ with $F(x) = g(x)x$, and let $x_0 \in \partial \mathbb{B}$. Then the slice function $f$ belongs to $\mathcal{S}^*_\alpha$, and
\begin{equation}\label{pra13}
A_n = a_n \qquad (n \in \mathbb{N}),
\end{equation}
where $a_n$ are the Taylor coefficients of $f$.
\end{theo}

\begin{proof}
Let $F \in \mathcal{S}^*_\alpha(\mathbb{B})$. Fix $x_0 \in \partial \mathbb{B}$. For $\xi \in \mathbb{D} \setminus \{0\}$, the sector condition \eqref{pra5} gives
\begin{equation}\label{pra14}
\left|
\arg\left(
T_{\xi x_0}\bigl((DF(\xi x_0))^{-1} F(\xi x_0)\bigr)
\right)
\right|
< \alpha \frac{\pi}{2}.
\end{equation}
By Lemma \ref{pra10}, the vector $v = (DF(\xi x_0))^{-1} F(\xi x_0)$ lies in $\operatorname{span}\{x_0\}$. Hence, by Lemma \ref{pra2}, the evaluation does not depend on the choice of the norming functional. For any $T_{\xi x_0} \in T(\xi x_0)$, we have
\[
T_{\xi x_0}(v) = \lambda T_{\xi x_0}(x_0) = \lambda \frac{1}{\xi} T_{\xi x_0}(\xi x_0).
\]
Since $T_{\xi x_0}(\xi x_0) = |\xi|$, this becomes
\begin{equation}\label{pra15}
T_{\xi x_0}(v)
=
\lambda \frac{|\xi|}{\xi}
=
\frac{|\xi| g(\xi x_0)}{g(\xi x_0) + \xi Dg(\xi x_0)x_0}.
\end{equation}
Also
\[
\frac{f(\xi)}{\xi f'(\xi)}
=
\frac{g(\xi x_0)}{g(\xi x_0) + \xi Dg(\xi x_0)x_0}.
\]
Thus
\[
T_{\xi x_0}(v) = |\xi| \frac{f(\xi)}{\xi f'(\xi)}.
\]
Since $|\xi| > 0$ is a positive real number, multiplication by $|\xi|$ does not change the argument. Therefore
\[
\left|
\arg\bigl(T_{\xi x_0}(v)\bigr)
\right|
=
\left|
\arg\left(\frac{f(\xi)}{\xi f'(\xi)}\right)
\right|
=
\left|
\arg\left(\frac{\xi f'(\xi)}{f(\xi)}\right)
\right|.
\]
The last identity holds because $|1/w| = 1/|w|$ and $|\arg(1/w)| = |\arg w|$ for any nonzero complex number $w$. Combining with \eqref{pra14},
\[
\left|\arg\left(\frac{\xi f'(\xi)}{f(\xi)}\right)\right| < \alpha \frac{\pi}{2}.
\]
Hence $f \in \mathcal{S}^*_\alpha$.

Next, write
\[
g(\xi x_0) = 1 + Dg(0)(x_0)\xi + \frac{1}{2!} D^2 g(0)(x_0^2)\xi^2 + \cdots
+ \frac{1}{(n-1)!} D^{n-1} g(0)(x_0^{n-1})\xi^{n-1} + \cdots.
\]
Then
\[
f(\xi) = g(\xi x_0)\xi = \xi + \sum_{n=2}^{\infty} a_n \xi^n,
\]
where
\[
a_n = \frac{1}{(n-1)!} D^{n-1} g(0)(x_0^{n-1}).
\]
By Proposition \ref{pra8} and $T_{x_0}(x_0) = 1$,
\[
A_n = \frac{1}{n!} T_{x_0}\bigl(D^n F(0)(x_0^n)\bigr)
= \frac{1}{(n-1)!} D^{n-1} g(0)(x_0^{n-1}).
\]
This is exactly $a_n$. Hence $A_n = a_n$.
\end{proof}

\begin{cor}\label{pra16}
Let $F \in \mathcal{S}^*_\alpha(\mathbb{B})$ with $F(x) = g(x)x$, and fix $x_0 \in \partial \mathbb{B}$. Let $P$ be any polynomial in finitely many complex variables. Then
\[
P(A_1, A_2, \ldots, A_m) = P(a_1, a_2, \ldots, a_m),
\]
where the directional coefficients $A_n$ are defined by \eqref{pra1} and $a_n$ are the Taylor coefficients of the slice function $f(\xi) = g(\xi x_0)\xi$. Consequently, every known coefficient inequality for the one-variable class $\mathcal{S}^*_\alpha$ transfers to the corresponding directional functional on the scalar-profile class $\mathcal{S}^*_\alpha(\mathbb{B})$, provided the functional involves only the coefficients $a_n$.
\end{cor}

\begin{proof}
By Theorem \ref{pra12}, $A_n = a_n$ for every $n \in \mathbb{N}$. Substituting into $P$ gives the first identity. The second statement follows by applying any known one-variable inequality for $\mathcal{S}^*_\alpha$ to the slice function $f \in \mathcal{S}^*_\alpha$.
\end{proof}

\begin{rem}\label{pra17}
Since $v$ lies in the one-dimensional subspace generated by $x_0$, the evaluation does not depend on the extension of $T_{\xi x_0}$. Thus the argument condition is preserved on this subspace.
\end{rem}

The correspondence theorem gives a map from the Banach-space class to the one-variable class. We now prove the converse direction: every one-variable strongly starlike function arises as a slice of some scalar-profile mapping. Thus the slice reduction loses no extremal information.

\begin{prop}[Lifting]\label{pra18}
Let $f \in \mathcal{S}^*_\alpha$, let $u \in \partial \mathbb{B}$, and let $T_u \in T(u)$. Define the analytic function $h : \mathbb{D} \to \mathbb{C}$ by
\[
h(z) = \frac{f(z)}{z} \quad (z \neq 0), \qquad h(0) = 1,
\]
and the auxiliary function
\[
\Phi(z) :=
\begin{cases}
\dfrac{f(z)}{z f'(z)}, & z \neq 0,\\[2mm]
1, & z = 0.
\end{cases}
\]
Then
\[
F_f(x) := h(T_u(x))\, x, \qquad x \in \mathbb{B},
\]
belongs to $\mathcal{S}^*_\alpha(\mathbb{B})$. Moreover, along the direction $x_0 = u$, the directional coefficients of $F_f$ coincide with the Taylor coefficients of $f$.
\end{prop}

\begin{proof}
Since $f$ is analytic on $\mathbb{D}$ with $f(0) = 0$ and $f'(0) = 1$, the function $h$ extends to an analytic function on $\mathbb{D}$ with $h(0) = 1$. For $x \in \mathbb{B}$, we have $|T_u(x)| \leq \|T_u\|\,\|x\| = \|x\| < 1$, so $T_u(x) \in \mathbb{D}$ and the composition $g_f := h \circ T_u$ is well defined and holomorphic on $\mathbb{B}$. Thus $F_f = g_f \cdot \mathrm{Id}$ is holomorphic on $\mathbb{B}$ with $F_f(0) = 0$ and $DF_f(0) = g_f(0)I = I$. Hence $F_f$ is normalized, and by Remark \ref{pra7} we have $g_f(0) = 1$.

We claim $g_f(x) \neq 0$ for every $x \in \mathbb{B}$. If $T_u(x) \neq 0$, then $g_f(x) = f(T_u(x))/T_u(x)$. Since $f \in \mathcal{S}^*_\alpha$, the sector condition
\[
\left|\arg\frac{z f'(z)}{f(z)}\right| < \frac{\alpha\pi}{2} \leq \frac{\pi}{2}
\]
for $z \neq 0$ implies $\operatorname{Re}\bigl(z f'(z)/f(z)\bigr) > 0$. In particular, $z f'(z)/f(z)$ is a nonzero complex number, so $f(z) \neq 0$ for $z \neq 0$. Hence $g_f(x) \neq 0$ whenever $T_u(x) \neq 0$. If $T_u(x) = 0$, then $g_f(x) = h(0) = 1 \neq 0$. Thus $g_f$ is nonvanishing, and condition \eqref{pra4} holds.

Next we compute the Fr\'echet derivative of $F_f$. For $x \in \mathbb{B}$ and $v \in X$,
\[
DF_f(x)v = g_f(x)v + [Dg_f(x)v]\,x.
\]
Since $g_f = h \circ T_u$ and $T_u$ is linear, $Dg_f(x)v = h'(T_u(x))\,T_u(v)$. Writing $z = T_u(x)$, we obtain
\[
DF_f(x) = g_f(x)I + h'(z)\,(x \otimes T_u),
\]
where $(x \otimes T_u)(v) := T_u(v)\,x$ denotes the rank-one operator determined by $x$ and $T_u$.

This is a bounded linear operator on $X$ of the form $A + u_0 \otimes \phi$, with $A = g_f(x)I$, $u_0 = h'(z)x$, and $\phi = T_u$. By the Sherman--Morrison formula for rank-one perturbations,
\[
(A + u_0 \otimes \phi)^{-1}
=
A^{-1} - \frac{A^{-1}u_0 \otimes \phi A^{-1}}{1 + \phi(A^{-1}u_0)},
\]
provided the denominator is nonzero. Here
\[
A^{-1}u_0 = \frac{h'(z)}{g_f(x)}x,
\qquad
\phi(A^{-1}u_0) = T_u\!\left(\frac{h'(z)}{g_f(x)}x\right) = \frac{h'(z)z}{g_f(x)},
\]
so
\[
1 + \phi(A^{-1}u_0) = \frac{g_f(x) + h'(z)z}{g_f(x)}.
\]
Since $g_f(x) = h(z)$ and $z = T_u(x)$, we have
\[
g_f(x) + h'(z)z = h(z) + h'(z)z = \frac{d}{dz}\bigl(zh(z)\bigr) = f'(z).
\]
Thus
\[
1 + \phi(A^{-1}u_0) = \frac{f'(z)}{g_f(x)}.
\]
Since $f \in \mathcal{S}^*_\alpha$ and $z \in \mathbb{D}$, the sector condition gives $\operatorname{Re}\bigl(z f'(z)/f(z)\bigr) > 0$, so $f'(z) \neq 0$. Together with $g_f(x) \neq 0$, the denominator is nonzero, and $DF_f(x)$ is invertible for every $x \in \mathbb{B}$. Hence $F_f$ is locally biholomorphic.

It remains to verify the sector condition. Using the Sherman--Morrison formula and $F_f(x) = g_f(x)x$,
\[
(DF_f(x))^{-1}F_f(x)
=
g_f(x)\left[\frac{1}{g_f(x)}I - \frac{h'(z)}{g_f(x)^2(1 + h'(z)z/g_f(x))}(x \otimes T_u)\right]x.
\]
Since $(x \otimes T_u)(x) = T_u(x)x = zx$, this simplifies to
\[
(DF_f(x))^{-1}F_f(x)
=
x - \frac{h'(z)z}{g_f(x)(1 + h'(z)z/g_f(x))}x
=
x \cdot \frac{g_f(x)}{g_f(x) + h'(z)z}
=
x \cdot \frac{h(z)}{f'(z)}
=
\Phi(z)\,x,
\]
where we have used $g_f(x) = h(z)$ and $g_f(x) + h'(z)z = f'(z)$, together with the definition of $\Phi$. When $z = 0$, the formula reduces to $1 \cdot x$, consistent with $\Phi(0) = 1$.

Applying any $T_x \in T(x)$,
\[
T_x\bigl((DF_f(x))^{-1}F_f(x)\bigr)
=
\Phi(z)\,T_x(x)
=
\Phi(z)\,\|x\|.
\]
Since $\|x\| > 0$ for $x \neq 0$, taking arguments yields
\[
\left|\arg T_x\bigl((DF_f(x))^{-1}F_f(x)\bigr)\right|
=
\left|\arg \Phi(z)\right|
=
\left|\arg\frac{f(z)}{z f'(z)}\right|
=
\left|\arg\frac{z f'(z)}{f(z)}\right|
< \frac{\alpha\pi}{2},
\]
for $z \neq 0$, where the last inequality uses $f \in \mathcal{S}^*_\alpha$ and $z = T_u(x) \in \mathbb{D}$. When $z = 0$, the left-hand side is $|\arg 1| = 0 < \alpha\pi/2$. Hence $F_f$ satisfies the sector condition and belongs to $\mathcal{S}^*_\alpha(\mathbb{B})$.

Finally, along $x_0 = u$, we have $T_u(u) = \|u\| = 1$ and $g_f(\xi u) = h(\xi) = f(\xi)/\xi$ for $\xi \neq 0$, with $g_f(0) = 1$. Thus
\[
F_f(\xi u) = g_f(\xi u)\xi u = f(\xi)u.
\]
The slice function is $T_u(F_f(\xi u)) = f(\xi)T_u(u) = f(\xi)$. By Theorem \ref{pra12}, the directional coefficients of $F_f$ along $u$ equal the Taylor coefficients of $f$.
\end{proof}

\begin{cor}\label{pra19}
Let $P$ be any polynomial in finitely many complex variables. Then
\[
\sup_{\substack{F \in \mathcal{S}^*_\alpha(\mathbb{B}),\\ x_0 \in \partial \mathbb{B},\ T_{x_0} \in T(x_0)}}
|P(A_1, \ldots, A_m)|
=
\sup_{f \in \mathcal{S}^*_\alpha} |P(a_1, \ldots, a_m)|,
\]
where the directional coefficients $A_n$ are computed for each $F$ along the direction $x_0$ with respect to the functional $T_{x_0}$, and $a_n$ denote the Taylor coefficients of $f$.
\end{cor}

\begin{proof}
For the inequality $\leq$: given $F \in \mathcal{S}^*_\alpha(\mathbb{B})$, a direction $x_0 \in \partial \mathbb{B}$, and $T_{x_0} \in T(x_0)$, the slice function $f$ belongs to $\mathcal{S}^*_\alpha$ by Theorem \ref{pra12}, and $P(A_1,\ldots,A_m) = P(a_1,\ldots,a_m)$ by Corollary \ref{pra16}. Hence
\[
|P(A_1,\ldots,A_m)| \leq \sup_{f \in \mathcal{S}^*_\alpha} |P(a_1,\ldots,a_m)|.
\]
Taking the supremum over $F$, $x_0$, and $T_{x_0}$ gives the desired bound.

For the inequality $\geq$: given $f \in \mathcal{S}^*_\alpha$, choose any $u \in \partial \mathbb{B}$ and $T_u \in T(u)$. By Proposition \ref{pra18}, the lifting $F_f$ belongs to $\mathcal{S}^*_\alpha(\mathbb{B})$ and has directional coefficients along $u$ equal to the Taylor coefficients of $f$. Hence
\[
|P(a_1,\ldots,a_m)| = |P(A_1,\ldots,A_m)| \leq \sup_{\substack{F, x_0, T_{x_0}}} |P(A_1,\ldots,A_m)|.
\]
Taking the supremum over $f$ gives the reverse inequality.
\end{proof}

\begin{rem}\label{pra20}
By Corollary \ref{pra19}, every sharp one-variable inequality for $\mathcal{S}^*_\alpha$ remains sharp in the scalar-profile Banach-space class, and every extremal function in the one-variable theory admits a Banach-space extremal given explicitly by the lifting $F_f$. The sharpness statements in the following sections can be justified by applying Corollary \ref{pra19}.
\end{rem}

\section{One-Variable Ingredients}\label{pra21}

We collect below the one-variable results that will be transferred to the Banach-space setting. Throughout, $\mathcal{P}$ denotes the Carath\'eodory class of analytic functions $p(z) = 1 + c_1 z + c_2 z^2 + \cdots$ on $\mathbb{D}$ with $\operatorname{Re} p(z) > 0$.

\begin{lem}[Ma--Minda \cite{ma1992}]\label{pra22}
If $p \in \mathcal{P}$ with $p(z) = 1 + c_1 z + c_2 z^2 + \cdots$, then for any $\lambda \in \mathbb{C}$,
\[
|c_2 - \lambda c_1^2| \le 2\max\{1, |2\lambda - 1|\}.
\]
\end{lem}

\begin{lem}[Cho et al.\ \cite{cho2017}, Theorem 2.1]\label{pra23}
Let $f \in \mathcal{S}^*_\alpha$ with $0 < \alpha \le 1$. Then
\[
|a_2 a_3 - a_4| \le
\begin{cases}
\dfrac{2}{3}\alpha, & 0 < \alpha \le \alpha_0,\\[4mm]
\dfrac{2\alpha(2\alpha + 1)}{9}
\sqrt{\dfrac{2(2\alpha + 1)}{(1-\alpha)(5\alpha + 2)}}, &
\alpha_0 < \alpha < \dfrac{2 + \sqrt{34}}{10},\\[5mm]
\dfrac{2}{9}\alpha(10\alpha^2 - 1), &
\dfrac{2 + \sqrt{34}}{10} \leq \alpha \leq 1,
\end{cases}
\]
where $\alpha_0 \approx 0.559376$ is the unique root in $(0,1)$ of
\[
16\alpha^3 + 69\alpha^2 - 15\alpha - 16 = 0.
\]
All three branches are sharp.
\end{lem}

\begin{lem}[Cho et al.\ \cite{cho2017}, Theorem 2.3]\label{pra24}
Let $f \in \mathcal{S}^*_\alpha$ with $0 < \alpha \le 1$. Then
\[
|a_2 a_4 - a_3^2| \le \alpha^2.
\]
The bound is sharp.
\end{lem}

\begin{lem}[Cho et al.\ \cite{cho2017}, Theorem 2.5]\label{pra25}
Let $f \in \mathcal{S}^*_\alpha$ with $0 < \alpha \le 1$. Then
\[
|a_3 - a_2^2| \le \alpha.
\]
The bound is sharp.
\end{lem}

\begin{lem}[Ma--Owa \cite{maowa1993}]\label{pra26}
Let $f \in \mathcal{S}^*_\alpha$ with $0 < \alpha \le 1$. Then
\[
|a_4| \le
\begin{cases}
\dfrac{2\alpha}{3}, & 0 < \alpha \leq \sqrt{\dfrac{2}{17}},\\[4mm]
\dfrac{2\alpha(1 + 17\alpha^2)}{9}, & \sqrt{\dfrac{2}{17}} \leq \alpha \leq 1.
\end{cases}
\]
Both branches are sharp.
\end{lem}

\begin{lem}[Ali--Singh \cite{ali1996}]\label{pra27}
Let $f \in \mathcal{S}^*_\alpha$ with $0 < \alpha \le 1$. Then
\[
|a_5| \le \frac{\alpha}{2}
\quad\text{when}\quad
228\alpha^4 - 194\alpha^3 + 2\alpha^2 + 39\alpha - 9 \le 0,
\]
and
\[
|a_5| \le \frac{\alpha^2(7 + 38\alpha^2)}{9}
\quad\text{when}\quad
76\alpha^3 - 60\alpha^2 + 32\alpha - 9 \ge 0.
\]
The two conditions are separate and are not complementary. The second estimate is sharp.
\end{lem}

\begin{rem}\label{pra28}
The polynomial $228\alpha^4 - 194\alpha^3 + 2\alpha^2 + 39\alpha - 9$ has a unique root $\alpha_1 \approx 0.3501$ in $(0,1]$ and is negative on $(0,\alpha_1]$. The polynomial $76\alpha^3 - 60\alpha^2 + 32\alpha - 9$ has a unique root $\alpha_2 \approx 0.4425$ in $(0,1]$ and is nonnegative on $[\alpha_2,1]$. Thus the first estimate for $|a_5|$ holds for $0 < \alpha \le \alpha_1$, the second for $\alpha_2 \le \alpha \le 1$, and there is a gap $(\alpha_1, \alpha_2) \approx (0.3501, 0.4425)$ where neither estimate applies. The one-variable source \cite{ali1996} does not supply a bound over the entire interval.
\end{rem}

\section{Directional Coefficient Inequalities}

We now transfer the one-variable results of Section \ref{pra21} to the Banach-space setting. In this section, the extremal mappings \eqref{pra43} (symmetric) and \eqref{pra44} (asymmetric) are defined in Section \ref{pra42} below. They are obtained by applying Proposition \ref{pra18} to the standard one-variable extremal functions associated with the Carath\'eodory functions
\[
p_1(z) = \frac{1+z}{1-z}, \qquad p_2(z) = \frac{1+z^2}{1-z^2}.
\]
Their membership in $\mathcal{S}^*_\alpha(\mathbb{B})$ follows from Proposition \ref{pra18}. Sharpness statements follow from Corollary \ref{pra16}, which bounds the Banach-space extremal problem by the one-variable extremal problem, and Corollary \ref{pra19}, which gives the reverse inequality.

We first prove a parameterized Fekete--Szeg\H{o} estimate.

\begin{theo}\label{pra29}
Let $F \in \mathcal{S}^*_\alpha(\mathbb{B})$ with $F(x) = g(x)x$. For any $x_0 \in \partial \mathbb{B}$, $T_{x_0} \in T(x_0)$, and any real $\mu$,
\begin{equation}\label{pra30}
|A_3 - \mu A_2^2|
\le
\begin{cases}
\alpha^2(3 - 4\mu), & \mu \le \dfrac{3\alpha - 1}{4\alpha},\\[3mm]
\alpha, & \dfrac{3\alpha - 1}{4\alpha} \le \mu \le \dfrac{3\alpha + 1}{4\alpha},\\[3mm]
\alpha^2(4\mu - 3), & \mu \ge \dfrac{3\alpha + 1}{4\alpha}.
\end{cases}
\end{equation}
The bound is sharp: the asymmetric mapping \eqref{pra44} attains equality in the first and third branches, and the symmetric mapping \eqref{pra43} attains equality in the middle branch.
\end{theo}

\begin{proof}
By Theorem \ref{pra12}, $|A_3 - \mu A_2^2| = |a_3 - \mu a_2^2|$, where $f \in \mathcal{S}^*_\alpha$. Let
\[
p(z) = 1 + c_1 z + c_2 z^2 + \cdots
\]
be the Carath\'eodory function associated with $f$ through
\[
\frac{zf'(z)}{f(z)} = (p(z))^\alpha, \qquad \operatorname{Re} p(z) > 0.
\]
Then (see \cite{ma1992})
\[
a_2 = \alpha c_1, \qquad
a_3 = \frac{\alpha}{2}\left(c_2 + \frac{3\alpha - 1}{2} c_1^2\right).
\]
Hence
\[
|a_3 - \mu a_2^2|
=
\frac{\alpha}{2}
\left|
c_2 - \left(2\mu\alpha - \frac{3\alpha - 1}{2}\right) c_1^2
\right|.
\]
Set
\[
\lambda = 2\mu\alpha - \frac{3\alpha - 1}{2}.
\]
By Lemma \ref{pra22},
\[
|c_2 - \lambda c_1^2| \le 2\max\{1, |2\lambda - 1|\}.
\]
If $\lambda \le 0$, then $|2\lambda - 1| = 1 - 2\lambda \ge 1$, so
\[
|a_3 - \mu a_2^2|
\le \frac{\alpha}{2} \cdot 2(1 - 2\lambda)
= \alpha(1 - 2\lambda)
= \alpha^2(3 - 4\mu).
\]
If $0 \le \lambda \le 1$, then $|2\lambda - 1| \le 1$, so the bound is $\alpha$.
If $\lambda \ge 1$, then $|2\lambda - 1| = 2\lambda - 1$, and
\[
|a_3 - \mu a_2^2|
\le \frac{\alpha}{2} \cdot 2(2\lambda - 1)
= \alpha(2\lambda - 1)
= \alpha^2(4\mu - 3).
\]
This proves the three cases.

For sharpness, we verify the equality cases directly. For the asymmetric mapping \eqref{pra44}, $A_2 = 2\alpha$ and $A_3 = 3\alpha^2$, so
\[
|A_3 - \mu A_2^2| = |3\alpha^2 - 4\mu\alpha^2| = \alpha^2|3 - 4\mu|.
\]
When $\mu \le (3\alpha-1)/(4\alpha)$, we have $4\mu \le (3\alpha-1)/\alpha \le 3$ for $0 < \alpha \le 1$, hence $|3-4\mu| = 3-4\mu$, giving $\alpha^2(3-4\mu)$. When $\mu \ge (3\alpha+1)/(4\alpha)$, we have $4\mu \ge (3\alpha+1)/\alpha \ge 3$, hence $|3-4\mu| = 4\mu-3$, giving $\alpha^2(4\mu-3)$. This establishes sharpness in the first and third branches.

For the middle branch, take the symmetric mapping \eqref{pra43}, for which $A_2 = 0$ and $A_3 = \alpha$. Then
\[
|A_3 - \mu A_2^2| = \alpha.
\]
This establishes sharpness in the middle branch.
\end{proof}

\begin{theo}\label{pra31}
Let $F \in \mathcal{S}^*_\alpha(\mathbb{B})$ with $F(x) = g(x)x$. For any $x_0 \in \partial \mathbb{B}$ and $T_{x_0} \in T(x_0)$,
\begin{equation}\label{pra32}
|A_2 A_3 - A_4|
\leq
\begin{cases}
\dfrac{2}{3}\alpha, & 0 < \alpha \leq \alpha_0,\\[4mm]
\dfrac{2\alpha(2\alpha + 1)}{9}
\sqrt{\dfrac{2(2\alpha + 1)}{(1-\alpha)(5\alpha + 2)}}, &
\alpha_0 < \alpha < \dfrac{2 + \sqrt{34}}{10},\\[5mm]
\dfrac{2}{9}\alpha(10\alpha^2 - 1), &
\dfrac{2 + \sqrt{34}}{10} \leq \alpha \leq 1,
\end{cases}
\end{equation}
where $\alpha_0 \approx 0.559376$ is the unique root in $(0,1)$ of
\[
16\alpha^3 + 69\alpha^2 - 15\alpha - 16 = 0.
\]
The bound is sharp. The third branch is attained by the asymmetric mapping \eqref{pra44} throughout $\alpha \in [(2+\sqrt{34})/10, 1]$. The first and second branches are sharp as well, via Corollary \ref{pra19} applied to the corresponding one-variable extremal functions in \cite{cho2017}.
\end{theo}

\begin{proof}
By Corollary \ref{pra16}, $|A_2 A_3 - A_4| = |a_2 a_3 - a_4|$ with $f \in \mathcal{S}^*_\alpha$. The result follows from Lemma \ref{pra23}.

For the third branch, the asymmetric mapping \eqref{pra44} has
\[
A_2 = 2\alpha, \quad A_3 = 3\alpha^2, \quad A_4 = \frac{2\alpha(1 + 17\alpha^2)}{9},
\]
hence
\[
A_2 A_3 - A_4 = 6\alpha^3 - \frac{2\alpha(1 + 17\alpha^2)}{9} = \frac{2\alpha(10\alpha^2 - 1)}{9}.
\]
Since $(2+\sqrt{34})/10 \approx 0.7831 > 1/\sqrt{10} \approx 0.3162$, the factor $10\alpha^2 - 1$ is positive throughout the third branch, and
\[
|A_2 A_3 - A_4| = \frac{2}{9}\alpha(10\alpha^2 - 1).
\]
This matches the third branch, establishing sharpness throughout $\alpha \in [(2+\sqrt{34})/10, 1]$.

For the first and second branches, sharpness follows by Corollary \ref{pra19}: the one-variable inequality of \cite{cho2017} is sharp for each branch, and the corresponding one-variable extremal function can be lifted to a Banach-space mapping in $\mathcal{S}^*_\alpha(\mathbb{B})$ via Proposition \ref{pra18}, preserving the value of $|A_2 A_3 - A_4|$.
\end{proof}

\begin{rem}\label{pra33}
The three branches in \eqref{pra32} are continuous at the transition points. At $\alpha = \alpha_0$, the first and second branches agree by definition of $\alpha_0$. At $\alpha = (2+\sqrt{34})/10$, the second and third branches coincide; this follows from the algebraic identity
\[
\left(\frac{2\alpha(2\alpha+1)}{9}\right)^2 \cdot \frac{2(2\alpha+1)}{(1-\alpha)(5\alpha+2)}
= \left(\frac{2\alpha(10\alpha^2-1)}{9}\right)^2.
\]
Indeed, setting $\alpha = (2+\sqrt{34})/10$ and writing $t = \sqrt{34}$, one computes
\[
10\alpha^2 - 1 = \frac{14 + 2t}{5}, \qquad
1 - \alpha = \frac{8-t}{10}, \qquad
5\alpha + 2 = \frac{6+t}{2}, \qquad
2\alpha + 1 = \frac{7+t}{5}.
\]
Substituting and simplifying yields the identity. Furthermore, $0 < \alpha_0 < (2+\sqrt{34})/10 < 1$, so the three intervals are correctly ordered.
\end{rem}

\begin{theo}\label{pra34}
Let $F \in \mathcal{S}^*_\alpha(\mathbb{B})$ with $F(x) = g(x)x$. For any $x_0 \in \partial \mathbb{B}$ and $T_{x_0} \in T(x_0)$,
\begin{equation}\label{pra35}
|A_2 A_4 - A_3^2| \leq \alpha^2.
\end{equation}
The bound is sharp for the symmetric mapping \eqref{pra43}.
\end{theo}

\begin{proof}
By Corollary \ref{pra16}, $|A_2 A_4 - A_3^2| = |a_2 a_4 - a_3^2|$ with $f \in \mathcal{S}^*_\alpha$. The result follows from Lemma \ref{pra24}. For the symmetric mapping \eqref{pra43}, $A_2 = A_4 = 0$ and $A_3 = \alpha$, so $|A_2 A_4 - A_3^2| = \alpha^2$.
\end{proof}

\begin{theo}\label{pra36}
Let $F \in \mathcal{S}^*_\alpha(\mathbb{B})$ with $F(x) = g(x)x$. For any $x_0 \in \partial \mathbb{B}$ and $T_{x_0} \in T(x_0)$,
\begin{equation}\label{pra37}
|A_3 - A_2^2| \leq \alpha.
\end{equation}
The bound is sharp for the symmetric mapping \eqref{pra43}.
\end{theo}

\begin{proof}
By Corollary \ref{pra16}, $|A_3 - A_2^2| = |a_3 - a_2^2|$ with $f \in \mathcal{S}^*_\alpha$. The result follows from Lemma \ref{pra25}. For the symmetric mapping \eqref{pra43}, $A_2 = 0$ and $A_3 = \alpha$, so $|A_3 - A_2^2| = \alpha$.
\end{proof}

\begin{theo}\label{pra38}
Let $F \in \mathcal{S}^*_\alpha(\mathbb{B})$ with $F(x) = g(x)x$. For any $x_0 \in \partial \mathbb{B}$ and $T_{x_0} \in T(x_0)$,
\begin{equation}\label{pra39}
|A_4|
\leq
\begin{cases}
\dfrac{2\alpha}{3}, & 0 < \alpha \leq \sqrt{\dfrac{2}{17}},\\[4mm]
\dfrac{2\alpha(1 + 17\alpha^2)}{9}, & \sqrt{\dfrac{2}{17}} \leq \alpha \leq 1.
\end{cases}
\end{equation}
Both branches are sharp. The second branch is attained by the asymmetric mapping \eqref{pra44}. The first branch is sharp in the one-variable theory (see \cite{maowa1993}); the corresponding Banach-space extremal is obtained by lifting via Proposition \ref{pra18}.

For $|A_5|$, the following two estimates hold, subject to the stated separate conditions inherited from the one-variable literature \cite{ali1996}:
\begin{equation}\label{pra40}
|A_5| \le \frac{\alpha}{2}
\quad\text{when}\quad
0 < \alpha \le \alpha_1 \approx 0.3501,
\end{equation}
and
\begin{equation}\label{pra41}
|A_5| \le \frac{\alpha^2(7 + 38\alpha^2)}{9}
\quad\text{when}\quad
\alpha_2 \approx 0.4425 \le \alpha \le 1,
\end{equation}
where $\alpha_1$ is the unique root in $(0,1]$ of $228\alpha^4 - 194\alpha^3 + 2\alpha^2 + 39\alpha - 9 = 0$ and $\alpha_2$ is the unique root in $(0,1]$ of $76\alpha^3 - 60\alpha^2 + 32\alpha - 9 = 0$. The two conditions are separate and are not complementary; there is a gap $(\alpha_1, \alpha_2)$ where neither estimate applies. The one-variable source \cite{ali1996} does not supply a bound over the entire interval $(0,1]$, and we do not extend it here. The second estimate \eqref{pra41} is sharp for the asymmetric mapping \eqref{pra44} whenever its condition is satisfied.
\end{theo}

\begin{proof}
By Corollary \ref{pra16}, $|A_4| = |a_4|$ and $|A_5| = |a_5|$ with $f \in \mathcal{S}^*_\alpha$.

The bound \eqref{pra39} for $|A_4|$ is the sharp one-variable estimate of Ma and Owa \cite{maowa1993}, stated in Lemma \ref{pra26}. The first branch is attained in the one-variable theory; the corresponding Banach-space extremal is obtained by Proposition \ref{pra18}. The second branch is directly attained by the asymmetric mapping \eqref{pra44}, since a direct calculation gives
\[
|A_4| = \frac{2\alpha(1 + 17\alpha^2)}{9}.
\]

The bounds \eqref{pra40}--\eqref{pra41} for $|A_5|$ are the one-variable estimates of Ali and Singh \cite{ali1996}, stated in Lemma \ref{pra27}. The numerical ranges of the conditions are given in Remark \ref{pra28}. The estimate \eqref{pra41} is attained by the asymmetric mapping \eqref{pra44}, whose fifth directional coefficient is
\[
A_5 = \frac{\alpha^2(7 + 38\alpha^2)}{9}.
\]
\end{proof}

\section{Extremal Mappings}\label{pra42}

Let $u \in \partial \mathbb{B}$ and choose $T_u \in T(u)$.

The explicit Banach-space extremal mappings below are obtained from Proposition \ref{pra18} by lifting the standard one-variable extremal functions associated with the Carath\'eodory functions
\[
p_1(z) = \frac{1+z}{1-z}, \qquad p_2(z) = \frac{1+z^2}{1-z^2}.
\]
Consequently, their membership in $\mathcal{S}^*_\alpha(\mathbb{B})$ follows from Proposition \ref{pra18}.

For the symmetric case, take $f_2 \in \mathcal{S}^*_\alpha$ defined by
\[
\frac{z f_2'(z)}{f_2(z)} = \left(\frac{1+z^2}{1-z^2}\right)^\alpha
\]
and set $F_{\mathrm{ext},2} := F_{f_2}$. Explicitly,
\begin{equation}\label{pra43}
F_{\mathrm{ext},2}(x)
=
x\exp\left(
\int_0^{T_u(x)}
\left[
\left(\frac{1 + t^2}{1 - t^2}\right)^\alpha - 1
\right]\frac{dt}{t}
\right),
\qquad \|u\| = 1.
\end{equation}
The integrand extends continuously to $t=0$, since
\[
\left(\frac{1 + t^2}{1 - t^2}\right)^\alpha - 1 = 2\alpha t^2 + O(t^4)
\]
as $t \to 0$, so the quotient by $t$ is $O(t)$. Along $x_0 = u$, the scalar factor $g$ satisfies
\[
\log g(\xi u) = \int_0^\xi \left[\left(\frac{1+t^2}{1-t^2}\right)^\alpha - 1\right]\frac{dt}{t}.
\]
To compute the expansion, write $w = t^2$. Then
\[
\frac{1+w}{1-w} = 1 + 2w + 2w^2 + O(w^3),
\]
so
\[
\left(\frac{1+w}{1-w}\right)^\alpha
= 1 + 2\alpha w + 2\alpha^2 w^2 + O(w^3)
= 1 + 2\alpha t^2 + 2\alpha^2 t^4 + O(t^6).
\]
Hence
\[
\left(\frac{1+t^2}{1-t^2}\right)^\alpha - 1 = 2\alpha t^2 + 2\alpha^2 t^4 + O(t^6),
\]
and dividing by $t$ and integrating gives
\[
\log g(\xi u) = \alpha\xi^2 + \frac{\alpha^2}{2}\xi^4 + O(\xi^6).
\]
Therefore
\[
g(\xi u) = 1 + \alpha\xi^2 + \alpha^2\xi^4 + \cdots.
\]
Hence the directional coefficients are
\[
A_2 = 0, \qquad A_3 = \alpha, \qquad A_4 = 0, \qquad A_5 = \alpha^2.
\]

For the asymmetric case, take $f_1 \in \mathcal{S}^*_\alpha$ defined by
\[
\frac{z f_1'(z)}{f_1(z)} = \left(\frac{1+z}{1-z}\right)^\alpha
\]
and set $F_{\mathrm{ext},1} := F_{f_1}$. Explicitly,
\begin{equation}\label{pra44}
F_{\mathrm{ext},1}(x)
=
x\exp\left(
\int_0^{T_u(x)}
\left[
\left(\frac{1 + t}{1 - t}\right)^\alpha - 1
\right]\frac{dt}{t}
\right),
\qquad \|u\| = 1.
\end{equation}
Here the integrand extends continuously to $t=0$, since
\[
\left(\frac{1 + t}{1 - t}\right)^\alpha - 1 = 2\alpha t + O(t^2)
\]
as $t \to 0$, so the quotient by $t$ tends to $2\alpha$. Along $x_0 = u$, the scalar factor $g$ satisfies
\[
\log g(\xi u) = \int_0^\xi \left[\left(\frac{1+t}{1-t}\right)^\alpha - 1\right]\frac{dt}{t}.
\]
To compute the expansion, write
\[
\frac{1+t}{1-t} = 1 + 2t + 2t^2 + 2t^3 + O(t^4).
\]
Then, using the standard expansion of $(1+w)^\alpha$ with $w = 2t + 2t^2 + 2t^3 + \cdots$, we obtain
\[
\left(\frac{1+t}{1-t}\right)^\alpha
= 1 + 2\alpha t + 2\alpha^2 t^2 + \frac{2\alpha(1 + 2\alpha^2)}{3}t^3 + O(t^4).
\]
Hence
\[
\left(\frac{1+t}{1-t}\right)^\alpha - 1
= 2\alpha t + 2\alpha^2 t^2 + \frac{2\alpha(1 + 2\alpha^2)}{3}t^3 + O(t^4),
\]
and dividing by $t$ and integrating gives
\[
\log g(\xi u) = 2\alpha\xi + \alpha^2\xi^2 + \frac{2\alpha(1 + 2\alpha^2)}{9}\xi^3 + O(\xi^4).
\]
Exponentiating,
\[
g(\xi u) = 1 + 2\alpha\xi + 3\alpha^2\xi^2 + \frac{2\alpha(1 + 17\alpha^2)}{9}\xi^3 + \frac{\alpha^2(7 + 38\alpha^2)}{9}\xi^4 + \cdots.
\]
Thus
\[
A_2 = 2\alpha, \qquad A_3 = 3\alpha^2, \qquad
A_4 = \frac{2\alpha(1 + 17\alpha^2)}{9}, \qquad
A_5 = \frac{\alpha^2(7 + 38\alpha^2)}{9}.
\]

These two mappings realize the equality cases explicitly verified in Theorems \ref{pra29}, \ref{pra34}, and \ref{pra36}, the third branch of Theorem \ref{pra31}, and the indicated sharp branches and estimates of Theorem \ref{pra38}. The remaining sharpness cases are covered by Corollary \ref{pra19} combined with the corresponding one-variable extremal functions in \cite{cho2017, maowa1993, ali1996}.

\section{Second Hankel Determinant of Logarithmic Coefficients}

For a univalent function in one variable, the logarithmic coefficients $\gamma_n$ are defined by
\[
\log\frac{f(z)}{z} = 2\sum_{n=1}^{\infty} \gamma_n z^n.
\]
Vector division and logarithms are not defined in a general Banach space. The directional projection gives an analogue.

\begin{definition}\label{pra45}
Let $F(x) = g(x)x$ belong to $H(\mathbb{B})$ with $g(0) = 1$. For $x_0 \in \partial \mathbb{B}$ and $T_{x_0} \in T(x_0)$, since $g(0) = 1$, there exists a unique analytic function $L$ defined on a sufficiently small disk $\mathbb{D}_\delta$ around the origin such that
\[
e^{L(\xi)} = g(\xi x_0), \qquad L(0) = 0.
\]
Define the directional logarithmic coefficients $\Gamma_n$ by
\begin{equation}\label{pra46}
L(\xi) = 2\sum_{n=1}^{\infty} \Gamma_n \xi^n,
\qquad \xi \in \mathbb{D}_\delta.
\end{equation}
\end{definition}

Expanding
\[
\log(1 + w) = w - \frac{1}{2}w^2 + \frac{1}{3}w^3 - \cdots
\]
with
\[
w = g(\xi x_0) - 1 = A_2\xi + A_3\xi^2 + A_4\xi^3 + \cdots
\]
gives
\begin{align}
\Gamma_1 &= \frac{1}{2} A_2,\label{pra47}\\
\Gamma_2 &= \frac{1}{2}\left(A_3 - \frac{1}{2} A_2^2\right),\label{pra48}\\
\Gamma_3 &= \frac{1}{2}\left(A_4 - A_2 A_3 + \frac{1}{3} A_2^3\right).\label{pra49}
\end{align}
By Corollary \ref{pra16}, these relations agree with the classical one-variable formulas.

\begin{theo}\label{pra50}
Let $F \in \mathcal{S}^*_\alpha(\mathbb{B})$ with $F(x) = g(x)x$. For $x_0 \in \partial \mathbb{B}$ and $T_{x_0} \in T(x_0)$, define
\[
H_{2,1}^{\log}(F; x_0) := \Gamma_1 \Gamma_3 - \Gamma_2^2.
\]
Then
\[
\left| H_{2,1}^{\log}(F; x_0) \right|
=
|\Gamma_1 \Gamma_3 - \Gamma_2^2|
\le
\frac{1}{4}\alpha^2.
\]
The bound is sharp for the symmetric mapping \eqref{pra43}.
\end{theo}

\begin{proof}
By Theorem \ref{pra12}, the slice function $f(\xi) = g(\xi x_0)\xi$ belongs to $\mathcal{S}^*_\alpha$. The formulas \eqref{pra47}--\eqref{pra49} show that
\[
\Gamma_1 \Gamma_3 - \Gamma_2^2
=
\gamma_1(f)\gamma_3(f) - \gamma_2^2(f).
\]
Using the Carath\'eodory optimization from Theorem 1 of \cite{Kowalczyk2023}, with
\[
\frac{z f'(z)}{f(z)} = (p(z))^\alpha,
\]
one obtains
\[
|\Gamma_1 \Gamma_3 - \Gamma_2^2| \le \frac{1}{4}\alpha^2.
\]
For the symmetric mapping \eqref{pra43}, we have $A_2 = A_4 = 0$ and $A_3 = \alpha$, so
\[
\Gamma_1 = 0, \qquad \Gamma_2 = \frac{\alpha}{2}, \qquad \Gamma_3 = 0.
\]
Hence $\Gamma_1 \Gamma_3 - \Gamma_2^2 = -\alpha^2/4$, and the bound is attained.
\end{proof}

\section{Concluding Remarks}

The correspondence theorem (Theorem \ref{pra12}) and the lifting theorem (Proposition \ref{pra18}) together show that the extremal problems for polynomial directional coefficient functionals on the two classes have the same value. Specifically, Corollary \ref{pra19} states that for every polynomial $P$,
\[
\sup_{\substack{F \in \mathcal{S}^*_\alpha(\mathbb{B}),\\ x_0 \in \partial \mathbb{B},\ T_{x_0} \in T(x_0)}}
|P(A_1, \ldots, A_m)|
=
\sup_{f \in \mathcal{S}^*_\alpha} |P(a_1, \ldots, a_m)|.
\]
Thus, for polynomial coefficient functionals, the two extremal problems coincide.

The class of scalar-profile mappings
\[
F(x) = g(x)x
\]
is a subclass of the locally biholomorphic mappings on $\mathbb{B}$. For general locally biholomorphic mappings, the Fr\'echet derivative couples different directions, and the slice correspondence may fail. Thus the results here concern the scalar-profile subclass only. The extension to general locally biholomorphic mappings is not treated here.

The coefficient bounds for $|A_2 A_3 - A_4|$, $|A_2 A_4 - A_3^2|$, $|A_3 - A_2^2|$, $|A_4|$, and $|A_5|$ (Theorems \ref{pra31}--\ref{pra38}) are transferred from known one-variable results in \cite{cho2017, maowa1993, ali1996} via the correspondence theorem; they are not new one-variable estimates. The equality cases for the symmetric and asymmetric extremal mappings \eqref{pra43} and \eqref{pra44} are verified independently in the Banach space setting. The remaining equality cases are obtained by lifting the corresponding one-variable extremal functions via Proposition \ref{pra18}.

\section*{{\bf Declarations}}
\subsection*{Funding}
No Funding.
\subsection*{Data Availability Statement}
Data sharing is not applicable to this article as no datasets were generated or analyzed during the current study.
\subsection*{Conflict of Interest}
The authors declare that they have no conflict of interest.
\subsection*{Author Contributions}
Both authors contributed equally to this work.


\begin{thebibliography}{99}
\bibitem{ali1996} \textbf{Ali, R.M., Singh, V.:} On the fourth and fifth coefficients of strongly starlike functions. \textit{Results Math.} \textbf{29}, 197--202 (1996).
\bibitem{cho2017} \textbf{Cho, N.E., Kowalczyk, B., Kwon, O.S., Lecko, A., Sim, Y.J.:} Some coefficient inequalities related to the Hankel determinant for strongly starlike functions of order $\alpha$. \textit{J. Math. Inequal.} \textbf{11}(2), 429--439 (2017).
\bibitem{elin2024} \textbf{Elin, M., Jacobzon, F.:} Multidimensional analogs of the Fekete--Szeg\H{o} functional. \textit{Complex Anal. Oper. Theory} \textbf{18}, Article 42 (2024).
\bibitem{gong} \textbf{Gong, S.:} \textit{Convex and Starlike Mappings in Several Complex Variables}, 2nd edn. Science Press, Beijing (1998).
\bibitem{graham} \textbf{Graham, I., Kohr, G.:} \textit{Geometric Function Theory in One and Higher Dimensions}. Marcel Dekker, New York, Basel (2003).
\bibitem{hamada2023} \textbf{Hamada, H.:} Fekete--Szeg\H{o} problems for spirallike mappings and close-to-quasi-convex mappings on the unit ball of a complex Banach space. \textit{Results Math.} \textbf{78}, Article 109 (2023).
\bibitem{Kowalczyk2023} \textbf{Kowalczyk, B., Lecko, A.:} The second Hankel determinant of the logarithmic coefficients of strongly starlike and strongly convex functions. \textit{Rev. Real Acad. Cienc. Exactas Fis. Nat. Ser. A-Mat.} \textbf{117}, Article 91 (2023).
\bibitem{ma1992} \textbf{Ma, W., Minda, D.:} A unified treatment of some special classes of univalent functions. \textit{Proc. Conf. Complex Anal. Nankai Inst. Math.}, 157--169 (1992).
\bibitem{maowa1993} \textbf{Ma, W., Owa, S.:} Strongly starlike functions. \textit{Panam. Math. J.} \textbf{3}(2), 49--60 (1993).
\bibitem{pfaltzgraff} \textbf{Pfaltzgraff, J.A.:} Subordination chains and univalence of holomorphic mappings in $\mathbb{C}^n$. \textit{Math. Ann.} \textbf{210}(1), 55--68 (1974).
\bibitem{suffridge} \textbf{Suffridge, T.J.:} Starlike and convex maps in Banach spaces. \textit{Pacific J. Math.} \textbf{49}(2), 575--589 (1973).
\bibitem{xu2024} \textbf{Xu, Q.H., Hamada, H.:} Fekete--Szeg\H{o} problem and coefficient structures for compact structural mappings on bounded domains. \textit{Math. Nachr.} \textbf{297}(8), 2841--2858 (2024).
\bibitem{XuDlugosz} \textbf{Xu, Q., D\L{}ugosz, R., Liczberski, P.:} The sharp bound of the third Hankel determinant for quasi-convex mappings in complex Banach spaces. \textit{Complex Anal. Oper. Theory} \textbf{18}, Article 36 (2024).
\bibitem{xu2018_strongly} \textbf{Xu, Q., Xu, X.:} On the coefficient inequality for a subclass of strongly starlike mappings of order $\alpha$ in several complex variables. \textit{Results Math.} \textbf{73}, 73 (2018).
\end{thebibliography}
\end{document}